\documentclass[11pt,a4paper,english]{article}

\usepackage[english]{babel}
\usepackage{url}
\usepackage[latin1]{inputenc}
\usepackage{amsmath, amssymb, bm, amsthm}
\usepackage{makeidx}
\usepackage{graphicx}
\usepackage{setspace}

\providecommand{\abs}[1]{\lvert#1\rvert}
\newcommand{{\STT}}{STT}
\newcommand{\STTs}{STTs}
\DeclareMathOperator{\Ran}{Ran}
\newtheorem{theorem}{Theorem}[section] 

\newtheorem{corollary}[theorem]{corollary}
\newtheorem{definition}{Definition}[section]
\theoremstyle{remark}
\newtheorem{remark}{Remark}[section]

\title{A copula-based method to build diffusion models with prescribed marginal and serial dependence} 

\author{Enrico Bibbona${}^{a}$, Laura Sacerdote${}^{a}$, Emiliano Torre${}^{b}$\\
${}^{a}$\small{Department of Mathematics ``G.Peano''}\\
\small{University of Torino, Torino, Italy}\\
${}^{b}$\small{Institute of Neuroscience and Medicine (INM-6) and}\\
\small{Institute for Advanced Simulation (IAS-6) and}\\
\small{JARA BRAIN Institute I, Jülich Research Centre, Jülich, Germany}\\
\small{enrico.bibbona@unito.it}\\ \small{laura.sacerdote@unito.it}\\ \small{e.torre@fz-juelich.de}\\
}
\date{}
\begin{document}
\maketitle

\begin{abstract}
This paper investigates the probabilistic properties that determine the existence of space-time transformations between diffusion processes. We prove that two diffusions are related by a monotone space-time transformation if and only if they share the same serial dependence. The serial dependence of a diffusion process is studied by means of its copula density and the effect of monotone and non-monotone space-time transformations on the copula density is discussed. This provides us a methodology to build diffusion models by freely combining prescribed marginal behaviors and temporal dependence structures. Explicit expressions of copula densities are provided for tractable models. A possible application in neuroscience is sketched as a proof of concept.
\end{abstract}

\section{Introduction} \label{Section1}

Monotone space-time transformations (STTs) have been shown to map diffusion processes into diffusion processes by Kolmogorov \cite{Kolmogorov} as early as 1931. Closed form expressions for the transition probability density function (pdf) are known in a limited number of cases, including the Wiener and the Cox-Ingersol-Ross (CIR) processes. The class of tractable models is  remarkably enlarged considering STTs, since the transition pdfs of two diffusion processes related by an STT can be calculated from one another. Necessary and sufficient conditions for the existence of STTs conserving the probability mass and mapping a diffusion process into a Wiener or a CIR process are known \cite{cherkasov,intoWiener,intoFeller}.  Such conditions prescribe relationships between the drift and the diffusion coefficients that do not have a probabilistic interpretation. Transformations not conserving the probability mass have been studied in \cite{bluman,intoWienerBarriere}, but they are beyond the aim of this paper.

Recently Kozlov \cite{kozlov} has given new insights on transformability of diffusions by STTs, proving that if an SDE admits a 3-dimensional symmetry algebra then the process can be transformed into a Brownian Motion. If there are only 2 symmetry generators, then the process is transformable into a CIR process. There is no general model for the 1- and 0-dimensional cases.
This new classification explains the special role of the Brownian Motion and the CIR processes among diffusions and why their Kolmogorov equations are easier to solve, but still it does not clarify which is the common \emph{probabilistic} structure that singles out the class of processes that can be transformed by an STT into one of the two mentioned models.
The first aim of this paper is to investigate the probabilistic ground that determines such transformability properties.

Our second goal concerns modeling. Many diffusion models of natural phenomena are formulated by considering a stationary diffusion process with an assigned marginal distribution (fitted from data or postulated a-priori) and specifying additional constraints on the dependence structure of the process. The authors of \cite{givencovariance} match the stationary distribution and the empirical autocovariance. In \cite{hyperbolicinfinance} an ad-hoc special form for the infinitesimal variance coefficient is proposed. Furthermore in \cite{juliesorensenMM, nonnormalOU} an STT is applied to an Ornstein-Uhlenbeck process in order to change its marginal distribution into other specific (e.g. bivariate or heavy tailed) ones. We investigate here the possibility to make this approach more general and systematic with a method that enables to specify the marginal behavior and the serial dependence of the diffusion separately.

To characterize the serial dependence of diffusions we use copulae.
Copulae have become a very common tool for modeling dependence in applied probability and statistics (cf. \cite{durante,joe2014,nelsen}). They have found application in many different fields ranging from finance and insurance \cite{luciano,embrechtsCopulas}, to reliability \cite{antonio, spizzichino}, stochastic ordering \cite{pellerey}, geophysics \cite{extremesNature}, neuroscience \cite{elisa, pillow, neuroscience1, jenison, Onken2009, laura}, statistics \cite{robertaMI} and many more. 

For every Markov process $X(\cdot)$ there exists a copula joining the random variables $X(s)$ and $X(t)$ that verifies the Chapman-Kolmogorov equation \cite{darsow}. Lageras \cite{lageras} underlines some limitations in the applicability of the classical families of copulae in the framework of stochastic processes showing, for example, that Fr\'echet copulae imply innatural Markov processes while Archimedean copulae are incompatible with the dependence of Markov chains.
Other objections to the use of copulae can be found in \cite{mikoschcopulae} and in the related discussions. According to \cite{mikoschcopulae} ``copulas do not really fit into the theory of stochastic processes and time series analysis''. This argument is certainly true to some extent (cf. also \cite{lageras}) but it seems to us too assertive. In this paper we  discuss instances in which copulae are helpful tools when modeling with diffusion processes. 

The properties of copulae of diffusions are illustrated in Section \ref{uniformized}, by means of the so called uniformized diffusion process.  
In Section \ref{CandSTT} we prove that two diffusions can be transformed into one other via  monotone STTs if and only if they share the same copula, up to a time stretching. Non-monotone transformations are also considered and their effect on the copula density is investigated.
The new understanding of monotone STTs as mappings which preserve the copula enables modeling applications that are discussed in Section \ref{ex} together with a set of examples. With the aim of going beyond the objection of \cite{lageras}, here we propose to use specific families of copulae selected from the tractable Wiener and CIR processes and by other diffusions related to them by STTs.

\section{Mathematical background and notations}

\subsection{Copulae}

Consider a pair of real-valued random variables (r.v.s) $X$ and $Y$ taking values in 
$\Ran X = [l_X, r_X]$, $\Ran Y = [l_Y, r_Y]$, respectively, where $\Ran$ is the range operator for functions. Let $F_{X,Y}\left( x,y\right) =P\left( X\leq x,Y\leq
y\right) $ be their joint distribution. $F_{X,Y}$ encodes all the probabilistic properties of $X$ and $Y$ in the same object: the marginal distributions of $X$ and $Y$, as well as any information on the dependence between the two r.v.s. (e.g. correlations of any order). However there are instances in which it is desirable to distinguish between marginal and joint properties of $(X,Y)$. Copulae are the appropriate tool to decouple the dependence properties of two r.v.s from their marginal behavior. 
In this section we mention definitions and properties of copula functions that will be used thereafter, referring to (\cite{nelsen}) for the proofs and for other facts in copula theory.
In this paper we consider Markov processes only, where the properties of the process at any time $t$ conditioned on the state of the process at a time $s<t$ are independent of the state at any time $r<s$. Thus, bivariate copulae are enough to characterize the temporal dependencies. For this reason we do not introduce the multivariate setup and in the remainder refer to bivariate copulae simply as copulae.
\begin{definition}
A copula is any function $C:\left[ 0,1\right] \times \left[
0,1\right] \rightarrow \left[ 0,1\right] $ such that:

\begin{itemize}
\item $\forall \;\; u,v\in \left[ 0,1\right]$, $C$ satisfies the boundary
conditions%
\[
C\left( u,0\right) =0,\quad C\left( 0,v\right) =0\quad C\left( u,1\right)
=u\quad C\left( 1,v\right) =v 
\]

\item $C$ is 2-increasing, i.e. $\forall $ $u,v,u^{\prime },v^{\prime }\in %
\left[ 0,1\right] $ such that $u\leq u^{\prime },v\leq v^{\prime }:$%
\[
V_{C}\left( \left[ u,u^{\prime }\right] \times \left[ v,v^{\prime }\right]
\right) :=C\left( u,u^{\prime }\right) -C\left( u^{\prime },v\right)
-C\left( u,v^{\prime }\right) +C\left( u,v\right) \geq 0. 
\]
\end{itemize}
\end{definition}

\begin{remark} Note that any copula is, by definition, a joint distribution over the unit square.
\end{remark}

\begin{theorem}[Sklar] Let $F_{X,Y}$, $F_X$, $F_Y$ be the distributions defined above. Then 
there exists a copula $C$ such that for all $(x,y) \in [l_X, r_X]\times[l_Y, r_Y]$, 
\begin{equation}
F_{X,Y}\left( x,y\right) = C\left( F_{X}\left( x\right) ,F_{Y}\left( y\right)\right).  
\label{Sklar_eq}
\end{equation}
If $F_X$ and $F_Y$ are continuous, $C$ is unique; otherwise it is uniquely 
determined on $\Ran F_X \times \Ran F_Y$. 
Conversely, if $C$ is a copula and $F_X$, $F_Y$ are distribution functions, 
then the function $F_{X,Y}$ defined by \eqref{Sklar_eq} is a joint distribution 
function with margins $F_X$ and $F_Y$.
\label{Sklar_thm}
\end{theorem}

Thus a copula is a function that combines, or ``couples", two margins to return a joint 
distribution. 

\begin{definition}

The r.v.s $U=F_{X}\left(X\right) $ and $V=F_{Y}\left( Y\right)$ are called uniformized 
r.v.s associated to $X$ and $Y$, respectively.
\end{definition}

\begin{definition}
The \emph{density} of a copula $C$ is the function $c$ defined by:
\begin{equation}
c\left( u,v\right) :=\frac{\partial ^{2}C\left( u,v\right) }{\partial u\partial v},  
\label{copula_density}
\end{equation}%
if the derivative exists. 
\end{definition}
Copula densities are probability density functions over the unit square.

\begin{theorem}\label{increasing}
Let $c_{X,Y}\left( u,v\right) $ be a copula density between the r.v.s $X$ and $Y$, 
and let $\alpha(\cdot) $ and $\beta(\cdot) $ be strictly monotone functions on 
$\Ran X$ and $\Ran Y$ respectively. Then \[c_{\alpha
\left( X\right) ,\beta \left( Y\right) }\left( u,v\right) =c_{X,Y}\left(
u,v\right). \]
If $\alpha(\cdot) $ and $\beta(\cdot) $ are strictly increasing then also $C_{\alpha
\left( X\right) ,\beta \left( Y\right) }\left( u,v\right) =C_{X,Y}\left(
u,v\right)$.
\label{monotone_transf}
\end{theorem}

\begin{remark}
 If $X$ and $Y$ are continuous, $F_X$ and $F_Y$ 
are strictly increasing. Thus, by virtue of Theorem \ref{monotone_transf} the copula of $(X,Y)$ is also the copula of $(U,V)$. It represents also the joint distribution of $(U, V)$, because the margins are uniform on the unit interval. 
\end{remark}

The following result connects copulae and conditional probabilities (cf. \cite{darsow}):

\begin{theorem}\label{conditionalDistribution}
Let $X_i$, $i=1,2$, be two r.v.s with margins $F_i$. Then%
\begin{equation}
\partial _{i} C\left( F_1\left(x_1\right), F_2\left(x_2\right) \right) = 
P\left( X_j\leq x_j\left\vert X_i=x_i\right. \right), \quad i,j=1,2, \; i \neq j
\end{equation}
where $\partial_i$, $i=1,2$, denotes the partial derivative with respect to the $i$-th argument.
\end{theorem}

The derivatives of a copula with respect to one of the arguments, when existing, are bivariate conditional distributions over the unit square. When coupling r.v.s uniformly distributed in the unit interval, the derivatives also represent the conditional distribution function of one r.v. with respect to the other.

\subsection{Diffusion processes}

Let $\left\{ X_t \right\} _{t\in[t_0,T_{I})}$ be a one
dimensional diffusion process taking values in the interval $I=\left( l,r\right) $ with $%
-\infty \leq l<r\leq +\infty$. Its sample paths are solution of the SDE%
\begin{equation}
\begin{array}{rcl}
dX_t &=&\mu \left( X_t ,t\right) dt+\sigma \left(
X_t ,t\right) dW_t  \\
X_{t_0} &=& x_{0}
\label{SDE} 
\end{array}
\end{equation}
for any $t_{0}\leq t<T_{I}$. Here $T_{I}$ is the first exit time from $I$ and $W_{t}=\left\{ W_t \right\} _{t\geq t_0}$
is a standard Wiener process.
We assume classical conditions (cf. \cite{ikeda}) on the drift $\mu \left( x,t\right) $ and the diffusion coefficient $\sigma \left( x,t\right)>0$ to ensure existence and unicity of the solution.

Let%

\[
F_{t|s}\left( x\left\vert y\right. \right) := P\left( X_t <x\left\vert X_s =y\right. \right) 
\]
be the transition distribution of $X$ and let
\begin{equation}
f_{t|s}\left( x\left\vert y\right. \right) =\frac{\partial F_{t|s}\left( x\left\vert y\right. \right) }{\partial x}, \quad s<t,
\label{pdf}
\end{equation}%
be the corresponding transition pdf. Whenever it does not generate confusion, we drop the initial conditions from the notation of the marginal distributions, denoting them by $F_{t}\left( x\right)$ instead of $F_{t|t_{0}}(x|x_{0})$.

When the
boundaries $l,r$ of $I$ are natural, according to Feller's
classification (cf. \cite{Karlin}), the function \eqref{pdf} is the unique solution of the
Kolmogorov backward equation%
\begin{equation}
\frac{\partial p}{\partial s}+\mu \left( y, s\right) \frac{%
\partial p}{\partial y}+\frac{\sigma ^{2}\left( y, s\right) }{2}%
\frac{\partial ^{2}p}{\partial y^{2}}=0  
\label{Kolmogorov_back}
\end{equation}%
with the final condition%
\begin{equation}
\lim_{s\uparrow t}p\left( x,t\left\vert y,s\right. \right)
=\delta \left( x-y\right).   
\label{final_dens}
\end{equation}%
Here $\delta \left( \cdot \right) $ indicates the Dirac delta. Furthermore, the transition distribution $F_{t|s}\left( x\left\vert y\right. \right)$ 
is the unique solution of (\ref{Kolmogorov_back}) when (\ref{final_dens}) is replaced with 
\begin{equation*}
\lim_{s\uparrow t}F_{t|s}\left( x\left\vert y\right. \right) =%
\boldsymbol{1}_{[0,\infty)}\left( x-y \right),
\end{equation*}%
where $\boldsymbol{1}_{[0,\infty)}(\cdot)$ denotes the Heaviside step function, i.e. the indicator function of the positive half-line.

The transition pdf $f$ also solves the Fokker Planck equation%
\begin{equation}
\frac{\partial p}{\partial t}+\frac{\partial }{\partial x}\left[ \mu \left(
x,t\right) p\right] -\frac{1}{2}\frac{\partial ^{2}}{\partial x^{2}}%
\left[ \sigma ^{2}\left( x,t\right) p\right] =0.
\label{Fokker_Plank}
\end{equation}%
When both boundaries are natural the solution corresponding to the initial condition
\begin{equation*}
\lim_{t\downarrow s}p\left( x,t\left\vert s,y\right. \right)
=\delta \left( x-y\right).
\end{equation*}
is unique. Otherwise further boundary conditions should be added to guarantee the uniqueness of the solution of
\eqref{Kolmogorov_back} or \eqref{Fokker_Plank}. 

\subsubsection{Copulae for diffusion processes}
For fixed times $t>s>t_0$, $X_s$ and $X_t$ are continuous r.v.s.  We are interested in investigating the structure of their dependence using copulae. Copulae of Markov processes have been extensively studied in \cite{darsow}. Here we focus on the special case when the Markov process is a diffusion.
We denote by $F_{s,t}(x,y)= P\left( X_s\leq x, X_t\leq y \right)$ the joint distribution of $X_s$ and $X_t$. The copula between $X_s$ and $X_t$ evaluated in $(u,v) \in [0,1]\times[0,1]$ is
\[C_{s,t}(u,v)=F_{s,t}[F^{-1}_{s} (u),F^{-1}_{t}(v)]\]
the associated copula density is
\[c_{s,t}(u,v)= \frac{f_{s,t}(F^{-1}_{s} (u), F^{-1}_{t} (v))}{f_{s}(F^{-1}_{s} (u))f_{t}(F^{-1}_{t} (v))} =\frac{f_{t|s}\left(F^{-1}_{t} (v)| F^{-1}_{s} (u)\right) }{f_{t}\left( F^{-1}_{t}(v)\right)},\]
and
\[C_{s,t}(u,v)=\int_{0}^{v}\int_{0}^{u}c_{s,t}(w,z)\,dw\,dz.\] 

\section{Uniformized diffusion processes} \label{uniformized}

In the context of diffusion processes, the serial dependence of the observations is much more commonly investigated in terms of transition distribution than in terms of joint distributions. The parallel concepts in the language of copulae are introduced below. 
\begin{definition}
We call \emph{uniformized transition distribution} of a diffusion process $\left\{ X_t \right\} _{t>t_{0}}$
between times $s$ and $t$, $t_0 < s < t$, the function 
\begin{equation}
C_{t|s}(v|u):= \frac{\partial}{\partial u} C_{s,t}(u,v)=\int_{0}^{v}c_{s,t}(u,z)\,dz.
\label{pluto}\end{equation}
\end{definition}
The \emph{uniformized transition distribution} of a diffusion process retains all the information about the serial dependence carried by the transition distribution of $X$ irrespective of its marginal properties. Due to Theorem \ref{conditionalDistribution} the uniformized transition distribution is a distribution function with respect to the variable $v$ having support on $[0,1]$, i.e. a transition distribution. To give a clearer understanding of its meaning we introduce the concept of \emph{uniformized process}.

\begin{definition}
Given a diffusion process $\left\{ X_t \right\} _{t>t_{0}}$, for any choice of an initial time $t_{1}>t_{0}$ we define 
\emph{uniformized process} associated to $X$ the stochastic process $\left\{ \widetilde{X}_t \right\} _{t\geq t_{1}}$
given by
\[
\widetilde{X}_t :=F_{t}(X_{t}) \quad \forall t\geq t_{1}>t_0.
\]
\end{definition}

\begin{remark}
The initial time $t_{1}$ for $\widetilde{X}_t$ is set to any time that strictly follows the initial time $t_{0}$ of $X$ because for any $t>t_{0}$ the marginal distribution $F_{t}(X_{t})$ is continuos. Note that this is not the case for $t=t_{0}$, when the probability mass is concentrated at $x_{0}$, and that would make the trajectories of $\widetilde{X}_t$ discontinuous at $t_{0}$. If $X$ is initialized at random with a continuous initial density supported in $J \subset I$ the extension $\widetilde{X}_t$ to $t_{0}$ becomes immediate.
\end{remark}

The \emph{uniformized process} $\widetilde{X}$ associated to $X$ takes values in $[0,1]$ and has uniform marginal distribution at each time. The joint probability distribution of two observations $\left( \widetilde{X}_s, \widetilde{X}_t \right)$ with $s,t\geq t_{1}$ is
\begin{align}\widetilde{F}_{s,t}(u,v)&= P\left( \widetilde{X}_s\leq u, \widetilde{X}_t\leq v \right)= P\left(X_s\leq F^{-1}_{s} (u), X_t\leq F^{-1}_{t}(v) \right)\notag\\
&=F_{s,t}[F^{-1}_{s} (u),F^{-1}_{t}(v)]=C_{s,t}(u,v)\notag\end{align}
and it coincides with the copula of $\left(X_s,X_t\right)$. By Theorem \ref{conditionalDistribution}, the associated transition probability distribution is given by
\begin{align}\widetilde{F}_{t|s}(v|u)=P\left( \widetilde{X}_s\leq u| \widetilde{X}_t= v \right)= P\left(X_s\leq F^{-1}_{s} (u)| X_t= F^{-1}_{t}(v) \right)=C_{t|s}(v|u)\label{transition_uniformized}\end{align}
for any $u,v \in [0,1]$ and any $t\geq s \geq t_{1}>t_{0}$.
Moreover, since at any given time the uniformized process is obtained by a strictly increasing transformation of $X$, in light of Theorem \ref{monotone_transf} it retains the same copula function of the process $X$. To conclude, $X$ and $\widetilde X$ share the same copula function which is also the joint distribution of $\widetilde X$.

\begin{theorem}
The transition pdf $\widetilde f_{t|s}(v|u)$ of the uniformized process $\widetilde{X}$ coincides with the joint copula density $c_{s,t}(u,v)$ of the process, i.e.
\begin{equation}\widetilde f_{t|s}(v|u)=c_{s,t}(u,v)=\frac{f_{t|s}\left(F^{-1}_{t} (v)| F^{-1}_{s} (u)\right) }{f_{t}\left( F^{-1}_{t}(v)\right)}\label{TransitionCopulaDensity}
\end{equation}
for any $u,v \in [0,1]$ and any $t\geq s \geq t_{1}>t_{0}$.
\end{theorem}

\begin{proof}
By definition $\widetilde f_{t|s}(v|u)=\frac{\partial}{\partial v}\widetilde{F}_{t|s}(v|u)$ and by equation \eqref{transition_uniformized} and equation \eqref{pluto}, the thesis follows.
\end{proof}


\begin{theorem}
Let equation \eqref{SDE} admit a unique solution $\left\{ X_t \right\}
_{t\geq t_0}$, having diffusion interval $I=\left( l,r\right)$.
The uniformized process $\left\{ \widetilde{X}_t \right\}
_{t> t_1}$ associated to $\left\{ X_t \right\}
_{t\geq t_1}$ is an It\^o diffusion process, initialized at 
$\widetilde{X}\left( t_{1}\right) = F_{t_{1}}\left( X_{t_{1}} \right)$ with a uniform distribution. 
Its drift $\widetilde{\mu}$ and diffusion 
coefficient $\widetilde{\sigma}$ are given by
\begin{equation}
\begin{array}{rcl}
\widetilde{\mu}(u,s)  &=& \left. \frac{\partial F_s\left(
x\right) }{\partial s}\right\vert _{x=F_s^{-1}\left( u \right) }+
\mu \left( F_{s\,}^{-1}\left( u \right) ,s\right) f_{s}\left( 
F_{s\,}^{-1}\left( u \right) \right) + \qquad \\
&& \hfill +\frac{1}{2}\sigma ^{2}\left( F_{s\,}^{-1}\left( u \right) ,s\right) 
f_{s\,}^{\prime }\left( F_{s\,}^{-1}\left( u \right) \right), \\[6pt]

\widetilde{\sigma}(u,s) &=& \sigma \left( F_{s\,}^{-1}\left( u \right)
,s\right) f_s\left( F_{s\,}^{-1}\left( u \right) ,s\right),
\end{array}
\label{uniformized_coefficients}
\end{equation}
where $f'_{s}\left( x\right)= \frac{\partial f_{s}\left( x\right) }{\partial x} =\frac{\partial^{2} F_{s}\left( x\right) }{\partial x^{2}}$.
\label{SDEuniformized_thm}
\end{theorem}

\begin{proof}
For any $t\geq t_{1}>t_{0}$ It\^o's formula yields%
\begin{multline*}
dF_t\left( X_t \right) =\left[ \frac{\partial F_{t\,}\left( X_t\right) }{\partial t}+\mu \left(
X_t ,t\right) f_{t}\left( X_t \right) +\frac{1}{2} \sigma ^{2}\left( X_t ,t\right) 
f_t^{\prime}\left( X_t \right) \right] dt + \\
 +\sigma \left( X_t ,t \right) f_{t}\left( X_t \right) dW_t
\end{multline*}

where $f_{t\,}^{\prime }\left( X_t ,t\right) =\left. \frac{%
\partial f_{t\,}\left( x\right) }{\partial x}\right\vert _{x=X_t }$. Substituting $F_t^{-1}\left( \widetilde{X}_t
\right) $ to $X_t $ in the last equality gives the thesis.
\end{proof}

\begin{remark} Diffusion processes sharing the same uniformized process constituite a class characterized by a unique transition copula.
\end{remark}

\begin{theorem}\label{kolm}
Let $\left\{ X_t \right\}
_{t\geq t_0}$ be a time-homogeneous diffusion process, having drift $\mu(x)$, diffusion coefficient $\sigma(x)$ and diffusion interval $I=\left( l,r\right)$ where $r$ and $l$ are natural boundaries. Let $F_{t}\left( x\right)$ be its marginal distribution. 
The uniformized transition distribution $C_{t|s}(v|u) $, for $t >s >t_{0}$, is the unique solution 
of the Kolmogorov backward equation%
\begin{equation}
\frac{\partial }{\partial s} C_{t|s}(v|u) + \widetilde{\mu}\left( F_{s}^{-1}\left( u\right) \right) 
\frac{\partial }{\partial u} C_{t|s}\left(v\left\vert u \right. \right) + \frac{1}{2}\widetilde{\sigma}^{2}\left( F_{s}^{-1}\left( u\right)\right)
\frac{\partial ^{2}}{\partial u^{2}}C_{t|s}(v|u) = 0,
\label{KolmogorovCopula}
\end{equation}
where $\widetilde{\mu}$ and $\widetilde{\sigma}$ are given by \eqref{uniformized_coefficients}, 
with the final condition%
\begin{equation}
\lim_{s\uparrow t }C_{t|s}\left( v \left\vert u\right. \right) =%
\boldsymbol{1}_{[0,\infty)}\left( v-u\right).  
\label{initialCopula}
\end{equation}
\label{Kolmogorov_uniformized}

\end{theorem}
\begin{proof}
Let us first consider the final condition. We can rewrite its left hand side as
\[
\lim_{s\uparrow t }C_{t \vert s}\left( v\vert u \right) = 
\lim_{s\uparrow t }\frac{\partial}{\partial u}C_{s,t}\left( u; v \right)=
\frac{\partial}{\partial u} \left( \lim_{s\uparrow t } C_{s,t}\left( u; v \right) \right).
\]
Note that the last equality holds because the derivative is performed with respect to 
$u$ and the limit operates on $s$. Furthermore,

\begin{eqnarray*}
\lim_{s\uparrow t } C\left( u,s; v,t \right) &=& \lim_{s\uparrow t } 
F_{X,Y}\left( F_s^{-1}\left( u\right) ,F_{t}^{-1}\left( v\right) \right) \\
&=&F_{s}\left( \min \left( F_{s}^{-1}\left( u\right) ,F_{s }^{-1}\left(
v\right) \right) \right) \\
&=&F_{s}\left( F_{s}^{-1}\left( \min \left( u,v\right) \right) \right) =\min
\left( u,v\right).
\end{eqnarray*}%
Note that the transformation $F_t$ maps the boundary points $l$l and $r$ for $X$ into the boundary points $0$ and $1$ for $\tilde X$, repsectively. $F_t$ is increasing and therefore does not alter the nature of the boundaries, which are natural for $X$ if and only if they are natural for $\tilde X$.

As far as the Kolmogorov backward equation is concerned we note that Eq. \eqref{KolmogorovCopula} follows from equation \eqref{conditionalDistribution} and Theorem \ref{SDEuniformized_thm}. Indeed,
Equation \eqref{transition_uniformized} ensures that $C_{t|s}(u|v) $ is
the transition distribution of the uniformized process. Theorem \ref{SDEuniformized_thm} guarantees that the uniformized process is an It\^o diffusion, and provides its infinitesimal coefficients. 
Hence its transition
distribution must be solution of the Kolmogorov backward equation having the same infinitesimal coefficients.
\end{proof}

\begin{remark}[\emph{Reflecting boundaries}] In Theorem \ref{kolm} we assumed that the boundaries $l,r$ are natural. However, the uniformized transition distribution $C_{t|s}(v|u) $ is still the unique solution of the Kolmogorov backward equation \eqref{KolmogorovCopula} with final condition \eqref{initialCopula} and reflection condition $\frac{\partial}{\partial u} C_{t|s}(v|u)|_{u=0,1}=0$ if one or both the boundaries are regular and the transition distribution $F_{t|s}(y|x)$ fulfills a reflection condition $\frac{\partial}{\partial x} F_{t|s}(y|x)|_{x\in \{r,l\} }=0$ at the regular boundary. Indeed $\frac{\partial}{\partial u} C_{t|s}(v|u)$ can be calculated directly, and by \eqref{TransitionCopulaDensity} and \eqref{transition_uniformized} it is easy to prove that $\frac{\partial}{\partial u} C_{t|s}(v|u)|_{u \in \{0,1\} }=0$ if and only if $\frac{\partial}{\partial x} F_{t|s}(y|x)|_{x \in \{r,l\} }=0$.
\end{remark}

\begin{remark} 
The transformation $F_t(\cdot)$ used to transform $X$ into $\tilde{X}$ depends on $x_0$ and $t_0$. Hence the drift and diffusion coefficient of $\tilde{X}$ depend on $x_0$ and $t_0$ as well. This might be misleadingly interpreted as an indication that the Markov property does not hold for $\widetilde{X}$. 
However $x_0$ and $t_0$ are parameters of $F_t(\cdot)$: changing either $x_0$ or $t_0$ affects both the initial conditions and the transformation.  As a consequence, for each $x_0$ and $t_0$ we get a different uniformized process $\widetilde{X}$, characterized by its diffusion equations. Each process that can be obtained from $X$ through $F_t(\cdot)$ verifies the Markov property.
\label{remarkark_markov}
\end{remark}

Equations \eqref{uniformized_coefficients} reveal that generally the uniformized process is not time homogeneous, 
even if such is the original process. The following result ensures that in the stationary regime the uniformized process becomes time homogeneous again.
\begin{corollary}
Let $\left\{ X_t \right\} _{t\geq t_0}$ be a time homogeneous process, 
admitting steady state density $g(x)$ and steady state distribution $G(x)$.
Asympotitcally, when the distribution of $X_t$ coincides with $G(x)$, 
the uniformized process $\widetilde{X}_t$, $t>s$, is a time-homogeneous 
diffusion process, having infinitesimal coefficients%
\begin{equation}
\begin{array}{l}
\widetilde{\mu}(u,s) = \widetilde{\mu}(u) = \mu \left( G^{-1}\left( u\right) \right) 
g \left( G ^{-1}\left( u\right) \right) + \frac{1}{2}\sigma ^{2} \left( G^{-1}
\left( u\right) \right) g^\prime \left(G ^{-1}\left( u \right) \right), \\[6pt]
\widetilde{\sigma}(u,s) = \widetilde{\sigma}(u) = \sigma \left(
G ^{-1}\left( u\right) \right) g \left( G ^{-1}\left( u\right) \right),
\end{array}
\label{coeff_steadyunif}
\end{equation}%
\label{uniformized_steadystate_coefficients}
where $g^\prime(z) := dg(z)/dz$.
\end{corollary}

\begin{proof}
The proof is analogous to that of Theorem \ref{SDEuniformized_thm}.
\end{proof}

Corollary \ref{uniformized_steadystate_coefficients} ensures that the coefficients of the uniformized processes associated to  stationary diffusions admit (finite) limits for $t\rightarrow \infty$.

\section{Copulae and space-time transformations}\label{CandSTT}

Already in 1931, Kolmogorov \cite{Kolmogorov} has shown that monotone STTs of the type
\begin{equation}\begin{cases} 
y =\psi (t,x)\\
\tau =\varphi (t),
\end{cases}
\label{spacetime}
\end{equation}
with Jacobian $J(t,x)=\frac{\partial \psi (t,x)}{\partial x}>0$ for every $x\in I$ and $\varphi(t)$ non-decreasing, map diffusion processes into diffusion processes (the case $J(t,x)<0$ is analogous). It is also well known (cf \cite{bluman}) that the only transformations mapping the Kolmogorov backward equation of a diffusion process into the Kolmogorov backward equation of a different diffusion process have the same structure of \eqref{spacetime}.
If an invertible transformation relates the two diffusion processes $X_{t}$ and $Y_{\varphi (t)}=\psi (t,X_{t})$ and the transition pdf of $X_{t}$ is known, the one of $Y_{\tau}$ can be expressed (as long as probability mass is conserved) by
\begin{equation}
f_{\tau|\tau_{0}}^{Y}(y|y_{0})=\frac{f_{t|t_{0}}^{X}(x|x_{0})} {J(t,x)}.
\end{equation}

In \cite{intoWiener} and \cite{intoFeller} necessary and sufficient conditions for the existence of such transformations are given when $X_t$ is a Wiener or a Cox-Ingersoll-Ross, respectively. These conditions require the drift and diffucion coefficient of the original process to verify  an equation that has no immediate probabilistic interpretation and hides the reasons why some processes can be transformed into others and some cannot.
As already mentioned in Section \ref{Section1}, one of the goals of the paper is to provide a direct probabilistic interpretation of transformability between diffusion processes. The following Theorems \ref{primo} and \ref{nonmonotone} accomplish this task by establishing the mathematical relation between copulae and {\STTs}.

\begin{theorem}\label{primo}
A diffusion process $X$ can be transformed into a diffusion process $Y$ via a monotone {\STT} \eqref{spacetime} if and only if they share the same copula density up to the time transformation, i.e.
\begin{equation}
c^{X}_{s,t}(u,v)=c^{Y}_{\varphi(s),\varphi(t)}(u,v),\label{samecopula}
\end{equation}
for any $t_{0}<s<t$ and $u,v\in [0,1]$.
\end{theorem}

\begin{proof}
If a monotone transformation maps $X$ into $Y$, Theorem \ref{increasing} implies \eqref{samecopula}.
On the other hand, if \eqref{samecopula} holds then the two uniformized processes $\widetilde X_{t}=F^{X}_{t}(X_{t})$ and $\widetilde Y_{\varphi(t)}=F^{Y}_{t}(Y_{t})$ have the same transition pdf and applying $[F^{Y}_{\varphi(t)}(\cdot)]^{-1}$ to $\widetilde X_{t}$ yields a process obeying the same law as $Y_{\varphi(t)}$.

Hence, the monotone {\STT}
\[\begin{cases}
y=\left[F^{Y}_{\varphi(t)}\right]^{-1}\hspace{-2mm} \left(F^{X}_{t}(x)\right)\\
\tau =\varphi (t),
\end{cases}
\]
maps $X_{t}$ into $Y_{\tau}$.

\end{proof}
Theorem \ref{primo} highlights that monotone {\STTs} only affect the marginal distributions of a process while preserving the copula (saved for the stretching of the time axis). Two processes that can be mapped into each other by a monotone STT share the same dependence structure up to the time stretching.

\begin{remark}
The explicit expression of the copula density of a diffusion process whose transition is known and that can be mapped into a simpler diffusion by a monotone STT can be derived either using equation \eqref{TransitionCopulaDensity} or applying Theorem \ref{primo}. The latter approach is sometimes advantageous as we shall see explicitly in Remark \ref{ZZZ} and Remark \ref{UUU}
\end{remark}

In many occasions it is of interest to consider more general transformations than \eqref{spacetime}. In particular, relaxing the monotonicity assumptions of $\psi(t,\cdot)$ by piecewise monotonicity determines losing the invertibility of the transformations and also the possibility of applying Theorem \ref{increasing} directly, but broadens the class of processes that can be taken into account. Interesting examples are shown in Section \ref{ex}.

\begin{theorem}\label{nonmonotone}
Let $X$ be a diffusion process with known transition pdf and initial conditions. Let also $\psi:\mathbb{R}^+\times I\rightarrow \mathbb{R}$ be an {\STT} as in \eqref{spacetime} such that, for every $t$, $J_{t}(x)=\frac{\partial\psi(t,x)}{\partial x}$ exists and is not vanishing except for at most a countable number or points of the diffusion interval $I$.
Under the assumption of conservation of the probability mass, the diffusion $Y_{\varphi(t)}:=\psi (t,X_{t})$ admits for any $t_{0}<s<t$ copula density function
\begin{align}{}&c^{Y}_{\varphi(s),\varphi(t)}(u,v)=\label{tesi}\\
 &=\hspace{-6mm}\displaystyle \sum_{
\substack{x:\,\psi(t,x)=\left[F^{Y}_{\varphi(t)}\right]^{-1}\hspace{-1mm}  (v)\\
z:\,\psi(s,z)=\left[F^{Y}_{\varphi(s)}\right]^{-1}\hspace{-1mm}(u)}} \hspace{-7mm}w\left(s,z,\left[F^{Y}_{\varphi(s)}\right]^{-1}\hspace{-1mm} (u)\right)w\left(t,x,\left[F^{Y}_{\varphi(t)}\right]^{-1}\hspace{-1mm}  (v)\right) \;c^{X}_{s,t}\left(F_s^{X}(z), F_{t}^{X}(x)\right),\notag\end{align} 
where
\begin{equation}
w(s,z,q):=\frac{\displaystyle\frac{f^{X}_{s}(z)}{\abs{J_{s}(z)}}}{\displaystyle\sum_{a: \psi(s,a)=q} \frac{f^{X}_{s}(a)}{\abs{J_{s}(a)}}}.
\label{doppioW}
\end{equation}
\end{theorem}

\begin{proof}
Let us start by deriving the transformation formula for the transition pdf under the hypothesis of conservation of the total probability mass.
\begin{align}
f^{Y}_{\varphi(t)|\varphi(s)}(y|q)&=\frac{f^Y_{\varphi(s),\varphi(t)}(q,y)}{f_{\varphi(s)}^Y (q)}=
\frac{\displaystyle\sum_{
\substack{x:\,\psi(t,x)=y\\
z:\,\psi(s,z)=q}}
\frac{f^{X}_{s,t}(z, x)}{\abs{J_{t}(x)J_{s}(z)}}}{\displaystyle\sum_{z: \psi(s,z)=q}\frac{ f^{X}_{s}(z)}{\abs{J_{s}(z)}}}=\notag \\
&=\displaystyle\sum_{z:\,\psi(s,z)=q} w(s, z,q) \sum_{x:\,\psi(t,x)=y} 
\frac{f^{X}_{t|s}(x|z)}{\abs{J_{t}(x)}},\label{questa}
\end{align}
where the weights $w(\cdot, \cdot,\cdot)$ are defined in $\eqref{doppioW}$.

Further applying the definition of copula density \eqref{TransitionCopulaDensity} and equation \eqref{questa}, we get

\begin{align}c^{Y}_{\varphi(s),\varphi(t)}(u,v)&=\frac{f^{Y}_{\varphi(s),\varphi(t)}\left(\left[F^{Y}_{\varphi(s)}\right]^{-1}\hspace{-1mm}  (u), \left[F^{Y}_{\varphi(t)}\right]^{-1}\hspace{-1mm} (v)\right)}{f^{Y}_{\varphi(s)}\left(\left[F^{Y}_{\varphi(s)}\right]^{-1}\hspace{-1mm}  (u)\right) f^{Y}_{\varphi(t)}\left(\left[F^{Y}_{\varphi(t)}\right]^{-1}\hspace{-1mm} (v)\right)}=\notag\\
 &=\frac{\displaystyle\sum_{
\substack{x:\,\psi(t,x)=\left[F^{Y}_{\varphi(t)}\right]^{-1}\hspace{-1mm} (v)\\
z:\,\psi(s,z)=\left[F^{Y}_{\varphi(s)}\right]^{-1}\hspace{-1mm} (u)}}\frac{f^{X}_{s,t}(z,x)} {\abs{J_s(z)J_t(x)}}}{\displaystyle\sum_{
\substack{x:\,\psi(t,x)=\left[F^{Y}_{\varphi(t)}\right]^{-1}\hspace{-1mm} (v)\\
z:\,\psi(s,z)=\left[F^{Y}_{\varphi(s)}\right]^{-1}\hspace{-1mm} (u)}}\frac{f^{X}_{s}(z)f^{X}_{t}(x)}{\abs{J_{s}(z)J_{t}(x)}}}.\label{orror}
\end{align}
After applying equation \eqref{TransitionCopulaDensity} again to get $c^{X}_{s,t}\left(F_s^{X}(z), F_{t}^{X}(x)\right)$ from the transition of $X$, equation \eqref{orror} is easily reorganized into equation \eqref{tesi}.
\end{proof}

\section{Modeling with copulae and transformations: some explicit expressions}\label{ex}

Diffusion models are often defined by assigning a specific marginal distribution and fixing the remaining degrees of freedom through different approaches (cf. Section \ref{Section1} for a detailed list of references). For instance, the authors of \cite{givencovariance} fix the autocovariance function in order to control the serial dependence of the process. In \cite{juliesorensenMM,nonnormalOU} a specific marginal distribution (heavy tailed or multimodal) is assigned by transforming the tractable stationary Ornstein-Uhlenbeck process via suitable {\STTs}. The novel results provided in Section \ref{CandSTT} make such approach more general and systematic. In this framework, the design of a diffusion model requires two steps: the choice of the copula of a diffusion process $X$ (regardless of its marginals) and the selection of the marginals. The choice of the copula is performed in accordance with the serial dependence of the phenomenon of interest among the copulae of tractable diffusions.  
The selection of new marginal distributions is also performed in accordance with the phenomenon to be modeled. Any continuous marginal distribution $F^{Z}_{t}(x)$ can be imposed by applying the transformation $Z_{t}=\left[F^{Z}_{t}\right]^{-1}\hspace{-1mm}\left(F^{X}_{t}(X_{t})\right)$, without altering the copula of the process. 

Whenever the transition pdf of $X$ is explicitly known, it is possible to simulate sample paths of $Z$ exactly by transforming those of $X$. Furthermore, it becomes possible to estimate the parameters of $Z$ by maximum likelihood.

In the next Subsections we present some examples of diffusion processes whose copula can be computed. We follow the classification by Kozlov \cite{kozlov} based on the dimensionality of the symmetry algebra of the corresponding SDE. Such dimensionality is preserved by {\STTs} and allows to group the most tractable diffusion models in two classes. The first class comprises those processes whose SDEs admit a 3-dimensional symmetry algebra. All such processes (cf. \cite{kozlov}) can be transformed into Brownian motions with or without boundary conditions. In Section \ref{Brelated} we present some models belonging to this class and  derive the corresponding copula densities. The second class contains SDE models whose symmetry algebra is 2-dimensional. They can be transformed into the CIR process with suitable boundary conditions. Section \ref{CIRrelated} introduces the copula of the CIR process and a few related models.
SDEs with $1$- or $0$-dimensional symmetry algebra are usually handled by numerical methods or by simulations and are not considered in the present paper.

\subsection{Brownian Motion and related processes}\label{Brelated}
The Brownian motion or Wiener process is a regular diffusion on the interval $(-\infty,\infty)$ with natural boundaries. Following equation \eqref{TransitionCopulaDensity} its copula density is
\begin{align}
c_{t|s}(v|u)&=\frac{\sqrt{t}}{\sqrt{t-s}}\frac{\phi \left(\frac{\sqrt{t}\Phi^{-1}(v)-\sqrt{s}\Phi^{-1}(u)}{\sqrt{t-s}}\right)}{\phi \left(\Phi^{-1}(v)\right)},\label{Bcopula}
\end{align}
which can be immediately recognized as a special instance of the Gaussian copula density (cf. \cite{nelsen}).
A necessary and sufficient condition for a process which solves the SDE \eqref{SDE} to be transformable into a Wiener process (and hence to admit a 3-dimensional symmetry algebra) is that its drift and diffusion coefficients $\mu(\cdot, \cdot)$ and $\sigma(\cdot, \cdot)$ satisfy the relation (cf. \cite{intoWiener,cherkasov, kozlov})
\[\frac{\mu(x,t)}{\sigma(x,t)}=\frac{\sigma_{x}(x,t)}{2}+\int\frac{\sigma_{t}(x,t)}{\sigma^{2}(x,t)}dx+c_{1}(t)\int\frac{dx}{\sigma(x,t)}+c_{2}(t)\]
where $c_{1}(t)$ and $c_{2}(t)$ are arbitrary functions of the time variable. If this condition is met, the {\STT} \eqref{spacetime} that realizes such mapping is
\[\begin{aligned}{}&\varphi(t)=\int_{t_{1}}^{t}\text{e}^{-2\int^{r}_{t_{0}} c_{1}(s)ds }dr\\
&\psi(t,x)=\sqrt{\varphi'(t)}\int_{x_0}^{x}\frac{dy}{\sigma(y,t)} +\int_{t_{2}}^{t}c_{2}(s)\sqrt{\varphi'(s)}ds\end{aligned}\]
where $x_{0}$, $t_{0}$, $t_{1}$, $t_{2}$ are arbitrary constants.

\subsubsection{Brownian motion with drift and geometric Brownian motion}
Brownian motion with drift $W_{t}=\mu t+\sigma B_{t}$ and Geometric Brownian motion $Y_{t}=\text{e}^{W_{t}}=\text{e}^{\mu t+ \sigma B_{t}}$ both result from purely spatial monotone transformations of a Brownian motion $B_{t}$. Hence, they both share the same Gaussian copula density as the Brownian motion. Importantly, this implies that the copula density does not depend on the drift and diffusion parameters $\mu$ and $\sigma$. 

Let us mention that while the transformation that maps $B_{t}$ into $W_{t}$ preserves the diffusion interval, the diffusion interval of $Y_{t}$ becomes $(0, \infty)$. Still, the nature of the boundaries according to Feller's classification remains unchanged (both boundaries are natural) since the transformation is monotone.

 \subsubsection{Ornstein-Uhlenbeck}
The Ornstein-Uhlenbeck process $X_{t}$ solves the SDE
 \begin{equation}dX_{t}=\left(-\alpha X_{t} +\beta\right) dt + \sigma d B_{t},\label{OUsde}\end{equation}
where $B_{t}$ is a standard Brownian motion and $\alpha, \, \beta \in \mathbb{R}$. The diffusion interval is $(-\infty,\infty)$ and both boundaries are natural. The process has a Gaussian transition pdf and also admits a Gaussian stationary distribution with mean $\frac{\beta}{\alpha}$ and variance $\frac{\sigma^{2}}{2\alpha}$ if $\alpha >0$. The coefficients of the copula process associated to $X_t$ can be derived from Eq. \eqref{coeff_steadyunif}. Equivalently, since a monotone {\STT} \eqref{spacetime} with $\varphi(t)=\frac{e^{2\alpha t}-1}{2\alpha}$, $\psi(t,x)=\frac{e^{\alpha t }}{\sigma}\left(x-\frac{\beta}{\alpha} \right)$ maps $X_t$ into a Brownian motion, the copula density is easily derived from Theorem \ref{primo}. According to \eqref{samecopula} it reads
\begin{equation} 
c^{\text{OU}}_{t|s}(v|u)=\frac{\sqrt{\text{e}^{-2\alpha t}-1}}{\sqrt{\text{e}^{-2\alpha t}-\text{e}^{-2\alpha s}}}\frac{\phi \left(\frac{\sqrt{\text{e}^{-2\alpha t}-1}\;\Phi^{-1}(v)-\sqrt{\text{e}^{-2\alpha s}-1}\;\Phi^{-1}(u)}{\sqrt{\text{e}^{-2\alpha t}-\text{e}^{-2\alpha s}}}\right)}{\phi \left(\Phi^{-1}(v)\right)}.
\label{OUcop}
\end{equation}
The same expression can also can be derived by directly plugging the transition pdf and the quantile function into equation \eqref{TransitionCopulaDensity}. 
\begin{remark} \label{ZZZ} If we determine the copula applying Theorem \ref{nonmonotone} instead of equation \eqref{TransitionCopulaDensity} we immediately note that the copula density does not depend on the parameters $\beta$ and $\sigma$ and that it depends on $\alpha$ only through the time transformation. The role of the parameter $\alpha$ can be interpreted as follows: for any given time $t$, the Gaussian copula $C_{0,t}(u,v)$ of the Ornstein Uhlenbeck process is close to the independent copula when $\alpha$ is very large and close to the copula of perfect positive dependence (also called the upper Fr\'echet bound) when $\alpha$ is small.
Parameter $\alpha$ encodes how short is the time range of the dependence, or the ``memory'' of the process.
\end{remark}

\subsubsection{Reflected Brownian motion}
Applying the purely spatial transformation $R_{t}=\abs{B_{t}}$ to a Brownian motion one gets the so-called reflected Brownian motion. The transformation maps both boundaries of the Brownian motion to $+\infty$, which remains a natural boundary. A lower regular boundary appears at $0$. The transition pdf solves the heat equation with a reflecting boundary condition in $0$. Since the transformation is piecewise monotone, we can apply Theorem \ref{nonmonotone} to get the copula density
\begin{equation}
c^{\text{RB}}_{s,t}(u,v)=\frac{\sqrt{t}}{2\sqrt{t-s}}\left[\frac{\phi \left(\frac{\sqrt{t}\Phi^{-1}(v)-\sqrt{s}\Phi^{-1}(u)}{\sqrt{t-s}}\right)}{\phi \left(\Phi^{-1}(v)\right)}+\frac{\phi \left(\frac{\sqrt{t}\Phi^{-1}(v)+\sqrt{s}\Phi^{-1}(u)}{\sqrt{t-s}}\right)}{\phi \left(\Phi^{-1}(v)\right)}\right],
\label{copulaRB}
\end{equation}
which is a mixture of Gaussian copulae.

\subsubsection{A special case of the CIR process}\label{special}
The process with SDE
\begin{equation}dX_{t}=(-\alpha X_{t} + \frac{\sigma^{2}}{4}) dt + \sigma\sqrt{X_{t}}\, dB_{t}\label{CIRspecial}\end{equation}
represents a special instance of the more general CIR process discussed in Section \ref{CIRrelated} below. The diffusion interval of $X_t$ is $(0,+\infty)$, where $0$ and $+\infty$ are a regular and a natural boundary, respectively (cf. \cite{fellerCIR}). In such a case the non-monotone {\STT} \eqref{spacetime} given by $\varphi(t)=\frac{\log(\alpha t + 1)}{\alpha}$ and $\psi(t,x)=\frac{\sigma^{2} x^{2}}{4(\alpha t + 1)}$ maps a Brownian motion $B_{t}$ into the solution $X_{t}$ of \eqref{CIRspecial}. The process can be equivalently obtained by applying the same transformation to a reflected Brownian motion $R_{t} =\abs{B_{t}}$ on the restricted domain $(0,\infty)$, where the transformation becomes monotone. Using the latter approach, we obtain the copula density of $X_{t}$ by a time transformation of the copula density of the reflected Brownian motion \eqref{copulaRB}
\[c^{\text{CIR special}}_{\varphi(s),\varphi(t)}(u,v) =c^{\text{RB}}_{s,t}(u,v),\]
which is once again a mixture of Gaussian copulae. 

\subsection{The CIR model and related processes} \label{CIRrelated}
The CIR process is the solution $\left\{ X_t \right\}_{t \geq t_0}$ of the SDE
\begin{equation}dX_{t}=(-\alpha X_{t} + \beta) dt + \sigma\sqrt{X_{t}}\, dB_{t}\label{CIRsde}\end{equation}
on $I=(0,\infty)$ with  $\beta>0$ and $B_{t}$ a standard Brownian motion. Let us denote by $x_{0}$ the initial condition at $t=0$. In the literature the CIR model is also named Feller process or linear drift, linear variance process.  If $\alpha>0$, the process admits a stationary gamma distribution. The lower boundary in $0$ is a singular point whose nature has been studied in \cite{fellerCIR}. If $\beta\geq\frac{\sigma^{2}}{2}$ the process never reaches zero (entrance boundary), while if $0<\beta<\frac{\sigma^{2}}{2}$ the process can reach zero and in order to solve the Kolmogorov equations a boundary condition needs to be imposed (we choose reflection). In both cases (cf. \cite{cir, rayleigh}) the transition pdf reads
\begin{equation}
f^{\text{CIR}}_{t|s}(x_{2}|x_{1})= c_{t-s}\, f_{\chi^2}\left(c_{t-s}\, x_{2}; \, \frac{4\beta}{\sigma^2}, \, c_{t-s}\, e^{-\alpha (t-s)} x_{1}\right), \label{densityCIR}
\end{equation}
where $c_{r}:=\frac{4\alpha }{\sigma^2 \left(1-e^{-\alpha (r)}\right)}$ and $f_{\chi^2}()$ is the non central chi-square density (cf. \cite{johnsonCUD2})

\[
f_{\chi^2}\left(z; \nu, \lambda \right) := \frac{1}{2}e^{-(z + \lambda)/2} \left( \frac{z}{\lambda}\right)^{\frac{\nu-2}{4}} I_{\frac{\nu-2}{2}}\left( \sqrt{\lambda z} \right).
\]
Here $I_{a}(\cdot)$ is the modified Bessel function of the first kind, defined by
\begin{equation*}
I_{a}(z) := \sum_{m=0}^\infty \frac{1}{m!\;\Gamma(m+a+1)} \left( \frac{z}{2}\right)^{2m+a}. 
\end{equation*}
Stable numerical algorithms that avoid direct evaluation of the Bessel function are available to evaluate the pdf \eqref{densityCIR}. The corresponding distribution function $F^{\text{CIR}}_{t_{2}|t_{1}}(x_{2}|x_{1})$ and its quantiles are also easily computable. A ready-to-use implementation is available in R, cf. \cite{Rmanual}.
A direct substitution into equation \eqref{TransitionCopulaDensity} provides a computable formula for the copula density of the CIR process:
\begin{multline} \label{Fellercopula}
c^{\text{Fe}}_{s,t}(u,v) = \frac{c_{t-s} f_{\chi^2} \left( \dfrac{c_{t-s}}{c_t} F_{\chi^2}^{-1}(v; \gamma, 0); \; \gamma, \; \dfrac{c_{t-s}}{c_s} e^{-\alpha(t-s)} F_{\chi^2}^{-1}(u; \gamma, 0) \right)}{c_{t} f_{\chi^2} \left( F_{\chi^2}^{-1}(v; \gamma, 0); \; \gamma, \; c_{t} e^{-\alpha t} x_{0} \right)},
\end{multline}
where $\gamma:=\frac{4\beta}{\sigma^2}$, $c_r:=\frac{2\alpha}{\sigma^2\left(  1-e^{-\alpha r}\right)}$, and $x_{0}$ is the initial condition at $t=0$.
However, as we prove in Remarks \ref{UUU} and \ref{DDD} below, such a method of calculation would hide much information about the effective dependence of the copula density on the parameters of the process.
\begin{remark}\label{UUU}
Despite the dependence of the transition pdf $f_{t_2|t_{1}}(x_{2}|x_{1})$, the marginal distribution $F_{t}(x)$ and of the quantile function $F^{-1}_{t}(p)$ of the CIR process on the three parameters $\alpha$, $\beta$, and $\sigma$, the copula density effectively depends on two of them only, namely $\alpha$ and $\gamma=\frac{4\beta}{\sigma^{2}}$.  A rigorous proof of the previous statement will be given in Remark \ref{DDD}. An interpretation of the role of the two parameters is given as follows. Parameter $\alpha$ retains the same role it has in the Ornstein-Uhlenbeck process, i.e. it is related with the range of the dependence of the process (the larger $\alpha$, the shorter the memory). Parameter $\gamma=\frac{4\beta}{\sigma^{2}}$ measures the ratio between the drift and the noise. If $\gamma\leq2$ the noise is strong enough to drive the process to 0, which is a reflecting barrier in our setting. If $\gamma=1$, in particular, then the process can be obtained from a reflected Brownian motion by a monotone {\STT} (see Section \ref{special}) and hence its copula is a mixture of Gaussian copulae. On the other hand, if $\gamma<2$ the drift prevails and the process does not reach the lower barrier. \end{remark}

\subsubsection{Rayleigh Process}
Any diffusion process \eqref{SDE} can be transformed into one with a constant unit diffusion coefficient through the {\STT}
\[
\psi(t,x)=\int\frac{dx}{\sigma(x)}.
\]

Such {\STT} for the CIR process $X_{t}$ of equation \eqref{CIRsde} becomes
\begin{equation}
Y_{t}=\frac{2\sqrt{X_{t}}}{\sigma}
\label{cir2ray}
\end{equation}
and one gets the Rayleigh process $Y_{t}$. Note that $\sigma$ in \eqref{cir2ray} is not any more the diffusion coefficient of a general diffusion as in the previous formula, but a parameter of equation  \eqref{CIRsde}.
The transformation is monotone and keeps the boundary points $0$ and $+\infty$ and their nature.
The diffusion process $Y_{t}$ solves the SDE
\begin{equation}dY_{t}=\left(\frac{a}{Y_{t}}+b Y_{t}\right)\, dt+ dB_{t}\label{RAY}\end{equation}
with $a=\frac{4\beta}{2\sigma^{2}}-\frac{1}{2}$, $b=-\frac{\alpha}{2}$, and $B_{t}$ a standard Brownian motion. The copula density of $Y_{t}$ is therefore identical to that of a CIR process.
\begin{remark}\label{DDD}It is now easy to prove the statement in Remark \ref{UUU} that the copula of the CIR process depends only on the two parameters $\alpha$ and $\gamma$. It immediately follows by the fact that it has to be the same copula as that of the Rayleigh process obtained by transformation \eqref{cir2ray}, whose two parameters $a$ and $b$ are indeed one to one functions of the original parameters $\alpha$ and $\gamma$.
\end{remark}
\subsubsection{Bessel Process}
Bessel processes of dimension $\delta+1$ are diffusions whose paths solve
\[dZ_{t}=\frac{\delta}{Z_{t}}dt + dB_{t}.\]
with $\delta>0$ and $B_{t}$ a standard Brownian motion. When $\delta$ is a non-negative integer, $Z_{t}$ has the same law as the euclidean norm of a $(\delta+1)$-dimensional Brownian motion.
The monotone transformations \eqref{spacetime} 
\begin{equation}\label{RAYtoBES}\varphi(t)=\frac{1 - \text{e}^{-2 b t}}{2 b} \qquad \psi(t,x)=x \text{e}^{-b t}\end{equation}
map a Rayleigh diffusion \eqref{RAY} into a Bessel process with parameter $\delta=a$. The diffusion interval remains $(0,\infty)$ and the nature of the boundaries is preserved. Using \eqref{cir2ray} and \eqref{RAYtoBES} we directly map a CIR process into a Bessel process with $\delta=\frac{\frac{4\beta}{\sigma^{2}}-1}{2}$, and hence the two copulae only differ for the time change
\[c^{\text{CIR}}_{s,t}(u,v)=c^{\text{BES}}_{\varphi(s),\varphi(t)}(u,v),\]
where the function $\varphi(\cdot)$ is given in equation \eqref{RAYtoBES}.

\subsection{Comparison of different copula densities}
In Section \ref{Brelated} and \ref{CIRrelated} we have introduced different copulae of diffusion processes. Apart from time changes, these copulae can be classified into three groups: Gaussian copulae (Brownian motion with or without drift, geometric brownian motion and Orstein-Uhlenbeck process), mixtures of Gaussian copulae (Reflected Brownian motion and the special case of the CIR process) and CIR-related copulae (CIR process in the general case, Rayleigh and Bessel processes).
In Figure \ref{plot1}, we present a visual comparison of the three copula densities.
 \begin{figure}[htb]
   \centering
   \includegraphics[width=\textwidth]{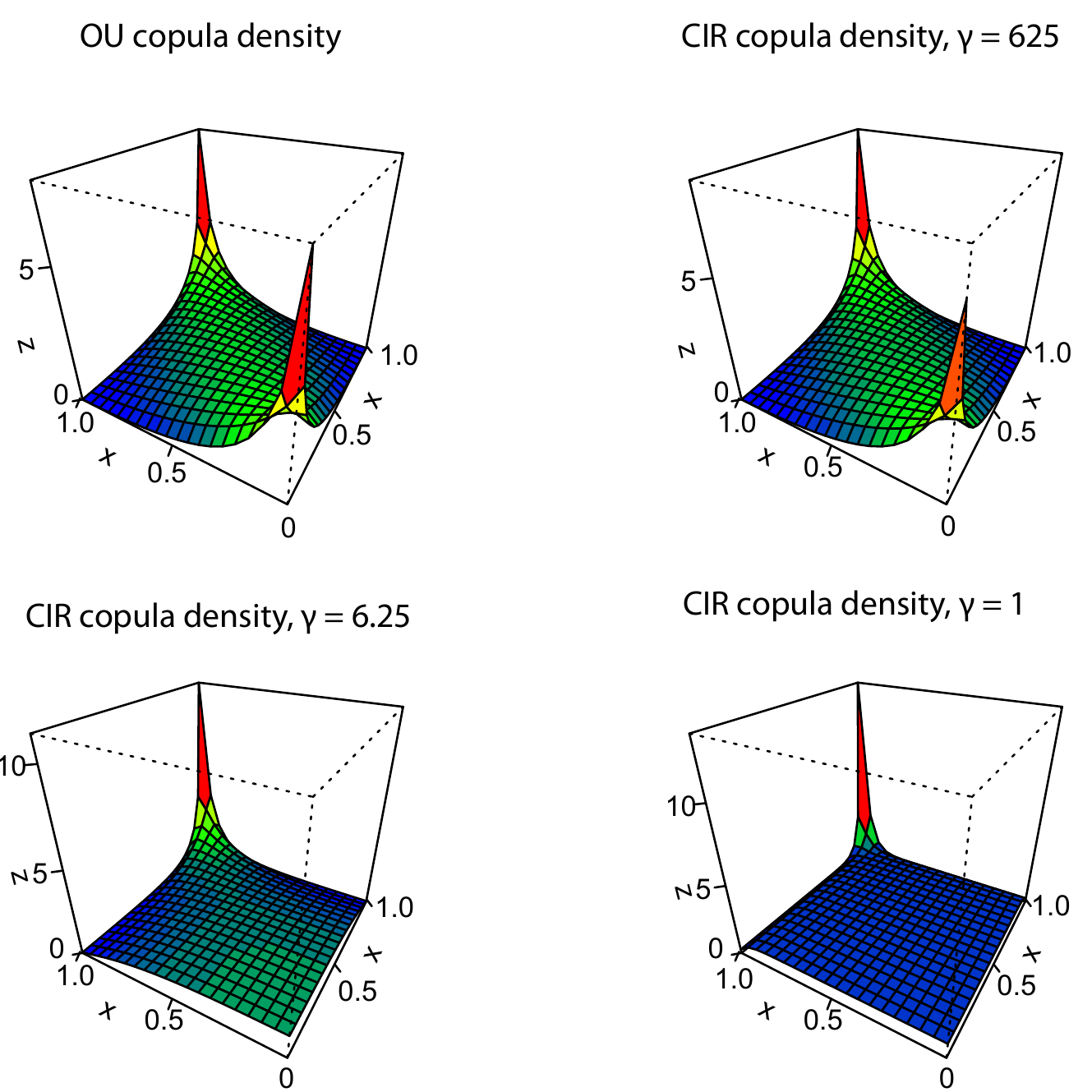}
   \caption{Comparison of copula densities.}\label{plot1}
\end{figure}

The underlying Ornstein-Uhlenbeck and CIR processes (also in the case when the CIR is a time-changed Brownian motion) are scaled to a common time scale through a suitable choice of the parameters of their SDEs \eqref{OUsde}, \eqref{CIRspecial}, and \eqref{CIRsde}. The scale of the dependence is established by the magnitude of the mean reverting parameter $\alpha$, which is kept fixed at $\alpha=0.1$. Other parameters common to all processes are  $s=30$ (stationary regimes), $t= 30.5$, $x_0=10$. The Gaussian copula density \eqref{OUcop} is plotted in the upper left corner. It is symmetric in the arguments and more concentrated around the diagonal. Two spikes are visible at $(0,0)$ and $(1,1)$. The CIR copula is plotted at three different values of the parameter $\gamma$ in the other panels of the Figure. For very large $\gamma$, the noise is very small with respect to the drift, and the CIR copula closely resembles the Gaussian one. In this regime the probability that the process approaches values close to 0 is so small that the effect of the barrier at $0$ is negligible. In particular, the upper right corner of the Figure shows the CIR copula density corresponding to $\gamma=625$. The smaller is $\gamma$, the larger is the noise compared to the drift. For $\gamma=6.35$ the copula becomes asymmetric, with a peak at $(1,1)$ and a flat region around $(0,0)$. Flat regions in the copula density correspond to regions of independence (note that the independent copula $C(u, v) = uv$ has a flat density $c(u,v)=1$, cf. \cite{nelsen}). In such regions the noise is strong enough compared to the drift to spread consecutive observations quickly, so that they appear almost independent. Indeed the stronger the noise is, the weaker is the dependence between small observations. Conversely, because the drift of the CIR process grows linearly with the value taken by the process while the noise grows only sublinearly, the dependence among large values remains strong. Therefore, large values of the process are more persistent than small values, especially in the presence of a strong noise.
The lower panels of the figure show the copula density for $\gamma=6.25$ and $\gamma=1$. The latter corresponds to the case when the CIR process is a time-changed reflected Brownian motion.
As apparent from the analysis of Figure \ref{plot1}, the three copula densities that we analyzed have different shapes which mark different properties of the associated dependence structures. The information a copula density conveys should be taken into account when modelling natural phenomena via diffusion processes.

\subsection{Application to neuronal modeling}
Ornstein-Uhlenbeck and CIR processes find extensive applications in neuroscience (cf. \cite{SacerdoteGiraudo2011, tamborrinojacobsen}), where they are used to model the temporal evolution of neuronal membrane potentials. Neural cells in the cortex communicate via fast discharges of electrical impulses called action potentials, or ``spikes''. Spikes arriving to a cell from input (``pre-synaptic'') neurons change the electrical membrane potential of the receiving cell by altering the difference between extra- and intra-cellular concentration of ions. 
In the integrate-and-fire model whenever the membrane potential hits a critical upper threshold -called firing threshold- a spike is generated and sent to receiving (``post-synaptic'') cells. Immediately after spike generation, the membrane potential is reset to a lower reset value. If a cell receives weak inputs from a large number of cells 
the sub-threshold evolution of its membrane potential until spike generation can be effectively modeled via continuous diffusions such as the Ornstein-Uhlenbeck and CIR (usually named Feller in the field) processes. If so, the spike times of the cell are first-passage times through the firing threshold. 

A comparison of the Ornstein-Uhlenbeck and Feller models based on the distribution of the interspike intervals has been carried out in \cite{lauraComparison}. The lower-bounded Feller model is often considered more realistic because compatible with the fact that neuronal membrane potentials cannot become arbitrarily negative. Such a consideration however only supports the marginal distribution of the Feller process compared to the Ornstein-Uhlenbeck one, and does not relate to the serial dependencies imposed by the two processes. The two processes are known to share the same autocovariance structure, which indeed only depends on the drift of the SDE. 
Realistic ranges for the parameters of the two diffusion process can be obtained by statistical estimation from experimental data of intracellular recordings, as done in \cite{bibbona,BibbonaLanskySirovich2010,Bibbona2008} and by theoretical reasoning (cf. \cite{lauraComparison}). In the case of the Ornstein-Uhlenbeck model \eqref{OUsde} the only parameter relevant to the copula is $\alpha$, which following the references above we take to be $0.1$ (see Figure \ref{plot1}, upper-left). In the case of the CIR process we also set $\alpha=0.1$, while realistic values for $\gamma$ range from 20 to 2500 ($\gamma=625$ in Figure \ref{plot1}, upper-right). As shown in Figure \ref{plot1}, in this parameter region the two copula densities are very similar. In light of these preliminary results, we suggest a new diffusion model for the temporal evolution of neuronal membrane potentials by combining the simpler copula of an Ornstein-Uhlenbeck process with the more realistic non central chi-square marginals of the CIR process. 

\section*{Acknowledgements}

We acknowledge financial support from the AMALFI project (Advanced Methodologies for the Analysis and management of the Future Internet, Compagnia di San Paolo and University of Torino), from the local research grant  2015 ``Stochastic modelling beyond diffusions'' of the University of Torino and from Indam-Gncs.


%
%
%
%

\bibliographystyle{apt}
\bibliography{bibl}

\end{document}